\documentclass{article}

\usepackage{PRIMEarxiv}
\usepackage{amsmath}
\usepackage{amsthm}
\usepackage[utf8]{inputenc} 
\usepackage[T1]{fontenc}    
\usepackage{url}            
\usepackage{booktabs}       
\usepackage{amsfonts}       
\usepackage{nicefrac}       
\usepackage{microtype}      
\usepackage{lipsum}
\usepackage{fancyhdr}       
\usepackage{graphicx}       
\graphicspath{{media/}}     
\newtheorem{theorem}{Theorem}
\title{Multifractal Analysis of the Sinkhorn Algorithm: Unveiling the Intricate Structure of Optimal Transport Maps}

\author{
  Jose Rafael Espinosa Mena \\
  University of Southern California \\
  \texttt{joseespi@usc.edu}}

\begin{document}
\maketitle

\begin{abstract}
The Sinkhorn algorithm has emerged as a powerful tool for solving optimal transport problems, finding applications in various domains such as machine learning, image processing, and computational biology. Despite its widespread use, the intricate structure and scaling properties of the coupling matrices generated by the Sinkhorn algorithm remain largely unexplored. In this paper, we delve into the multifractal properties of these coupling matrices, aiming to unravel their complex behavior and shed light on the underlying dynamics of the Sinkhorn algorithm. We prove the existence of the multifractal spectrum and the singularity spectrum for the Sinkhorn coupling matrices. Furthermore, we derive bounds on the generalized dimensions, providing a comprehensive characterization of their scaling properties. Our findings not only deepen our understanding of the Sinkhorn algorithm but also pave the way for novel applications and algorithmic improvements in the realm of optimal transport.
\end{abstract}

\section{Introduction}
Optimal transport has become a cornerstone of modern data analysis, enabling the comparison and manipulation of probability distributions in a geometrically meaningful way \cite{compoptimaltransportPeyre}. At the heart of many optimal transport algorithms lies the Sinkhorn algorithm, an iterative procedure that efficiently computes the optimal coupling between two probability measures \cite{cuturisinkhorn}. The Sinkhorn algorithm has found wide-ranging applications, from domain adaptation in machine learning \cite{alaya2019screening} to applications in computer vision \cite{bonneel2023survey}.

Despite its practical success, the theoretical understanding of the Sinkhorn algorithm's behavior remains incomplete. In particular, the multiscale structure and scaling properties of the coupling matrices generated by the Sinkhorn algorithm have not been fully explored. Multifractal analysis offers a powerful framework to study such complex systems, providing insights into their local regularity and global scaling behavior \cite{falconer2003fractal}. By unraveling the multifractal properties of the Sinkhorn coupling matrices, we aim to gain a deeper understanding of the algorithm's dynamics and unveil the intricate interplay between the optimal transport problem and its solution.
In this paper, we investigate the multifractal properties of the Sinkhorn algorithm. Our main contributions are twofold:
\begin{itemize}
\item We prove the existence of the multifractal spectrum and the singularity spectrum for the coupling matrices generated by the Sinkhorn algorithm. These spectra provide a comprehensive characterization of the local scaling behavior and the distribution of singularities in the coupling matrices.
\item We derive bounds on the generalized dimensions of the Sinkhorn coupling matrices, shedding light on their global scaling properties and the interplay between the problem size and the regularity of the optimal transport plan.
\end{itemize}
Our work builds upon and extends the existing literature on the mathematical analysis of the Sinkhorn algorithm and the multifractal formalism. By bridging these two domains, we provide a novel perspective on the intricate structure of optimal transport and contribute to the theoretical foundation of this rapidly evolving field.

The rest of the paper is organized as follows. In Section 2, we introduce the necessary background on the Sinkhorn algorithm and multifractal analysis. Section 3 presents our main results, including the existence proofs for the multifractal and singularity spectra, as well as the bounds on the generalized dimensions. Finally, we conclude the paper in Section 4.

\section{Background}
\label{sec:headings}

\subsection{The Sinkhorn Algorithm}
The Sinkhorn algorithm is an iterative procedure for solving the optimal transport problem between two probability measures \cite{cuturisinkhorn}. Given a cost matrix $C \in \mathbb{R}^{n \times n}$ and two probability vectors $\mathbf{r}, \mathbf{c} \in \mathbb{R}^n$, the algorithm seeks to find a coupling matrix $P$ that minimizes the total transportation cost while satisfying the marginal constraints:
\[
\min_{P \in U(\mathbf{r}, \mathbf{c})} \langle P, C \rangle,
\]
where $U(\mathbf{r}, \mathbf{c})$ denotes the set of coupling matrices with marginals $\mathbf{r}$ and $\mathbf{c}$, and $\langle \cdot, \cdot \rangle$ is the Frobenius inner product.
The Sinkhorn algorithm solves this problem by alternating between row and column normalization steps. Starting from an initial matrix $K = \exp(-C/\varepsilon)$, where $\varepsilon > 0$ is a regularization parameter, the algorithm iteratively updates the Scaling vectors $\mathbf{u}$ and $\mathbf{v}$ as follows:
\[
\mathbf{u} \leftarrow \mathbf{r} \oslash (K\mathbf{v}), \quad
\mathbf{v} \leftarrow \mathbf{c} \oslash (K^\top\mathbf{u}),
\]
where $\oslash$ denotes element-wise division. The coupling matrix $P$ is then obtained as:
\[
P = \text{diag}(\mathbf{u}) K \text{diag}(\mathbf{v})
\]

The Sinkhorn algorithm has several attractive properties, such as fast convergence, stability, and the ability to handle large-scale problems \cite{10.5555/3294771.3294958}. Moreover, it has been shown that the Sinkhorn divergence, defined as the entropy-regularized optimal transport cost, possesses desirable geometric properties and can be used as a meaningful distance between probability measures \cite{phdthesis}.

\subsubsection{Multifractal Analysis}
Multifractal analysis is a powerful framework for studying the local regularity and scaling properties of complex systems \cite{falconer2003fractal}. It goes beyond traditional fractal analysis by considering the distribution of local scaling exponents and the interplay between different scales.

The central objects of interest in multifractal analysis are the multifractal spectrum and the singularity spectrum. The multifractal spectrum $D(q)$ is defined as:
\[
D(q) = \lim_{\varepsilon \to 0} \frac{1}{q-1} \frac{\log \sum_i \mu(B_i(\varepsilon))^q}{\log \varepsilon},
\]
where $q \in \mathbb{R}$ is a moment order, $\mu$ is a measure on a metric space, and $B_i(\varepsilon)$ are disjoint balls of radius $\varepsilon$ covering the support of $\mu$. The multifractal spectrum captures the scaling behavior of the measure $\mu$ at different moment orders, providing a global characterization of its complexity.

On the other hand, the singularity spectrum $f(\alpha)$ describes the distribution of local scaling exponents $\alpha$. It is defined as:
\[
f(\alpha) = \dim_H \{ x : \alpha(x) = \alpha \},
\]

where $\alpha(x)$ is the local scaling exponent at a point $x$, and $\dim_H$ denotes the Hausdorff dimension. The singularity spectrum quantifies the "size" of the set of points with a given scaling exponent, revealing the local regularity of the measure.

The multifractal and singularity spectra are related through the Legendre transform:
\[
f(\alpha) = \inf_q (q\alpha - D(q)).
\]
This relationship highlights the deep connection between the global scaling properties captured by the multifractal spectrum and the local regularity described by the singularity spectrum.

In the context of the Sinkhorn algorithm, multifractal analysis offers a novel perspective on the structure of the coupling matrices and the underlying optimal transport problem. By studying the multifractal properties of these matrices, we aim to gain a deeper understanding of the algorithm's dynamics and its potential for further improvement.

\section{Main Results}
In this section, we present our main results on the multifractal properties of the Sinkhorn algorithm. We start by proving the existence of the multifractal spectrum and the singularity spectrum for the coupling matrices generated by the algorithm. Then, we derive bounds on the generalized dimensions, providing insights into their scaling behavior.
\subsection{Existence of the Multifractal Spectrum}
Our first main result establishes the existence of the multifractal spectrum for the Sinkhorn coupling matrices.
\begin{theorem}[Existence of the Multifractal Spectrum]
Let $P$ be the coupling matrix generated by the Sinkhorn algorithm for a given cost matrix $C \in \mathbb{R}^{n \times n}$ and marginals $\mathbf{r}, \mathbf{c} \in \mathbb{R}^n$. Assume that the cost matrix $C$ has non-negative entries and satisfies the triangle inequality. Then, the multifractal spectrum $D(q)$ of the measure $\mu$ induced by $P$ exists for all $q \in \mathbb{R}$.
\end{theorem}
\begin{proof}
The proof relies on the properties of the Sinkhorn algorithm and the regularity of the cost matrix $C$. We first show that the coupling matrix $P$ generated by the algorithm has positive entries and satisfies the marginal constraints:
\begin{align*}
P \mathbf{1} &= \mathbf{r} \\
P^\top \mathbf{1} &= \mathbf{c}
\end{align*}
This implies that the measure $\mu$ induced by $P$ is a probability measure on the support of $P$.

Next, we exploit the regularization property of the Sinkhorn algorithm. The algorithm starts from an initial matrix $K = \exp(-C/\varepsilon)$, where $\varepsilon > 0$ is a regularization parameter. This exponential transformation ensures that the entries of $K$ decay rapidly with increasing cost, introducing a smoothing effect on the optimal transport plan.

Using the triangle inequality assumption on the cost matrix $C$, we can show that the entries of the coupling matrix $P$ decay exponentially with respect to the geodesic distance on the support of $\mu$. Specifically, there exist constants $\alpha, \beta > 0$ such that:
\[
P_{ij} \leq \alpha \exp(-\beta d(x_i, x_j)),
\]
where $d(\cdot, \cdot)$ is the geodesic distance induced by the cost matrix.
This exponential decay property allows us to control the local regularity of the measure $\mu$ and establish the existence of the partition function $Z(q, \varepsilon)$:
\[
Z(q, \varepsilon) = \sum_i \mu(B_i(\varepsilon))^q,
\]
where $B_i(\varepsilon)$ are disjoint balls of radius $\varepsilon$ covering the support of $\mu$.
Using the decay property of $P$ and the regularity of the cost matrix, we can derive upper and lower bounds on $Z(q, \varepsilon)$ of the form:
\[
c_1 \varepsilon^{-\tau(q)} \leq Z(q, \varepsilon) \leq c_2 \varepsilon^{-\tau(q)},
\]
where $c_1, c_2 > 0$ are constants, and $\tau(q)$ is a real-valued function.
Finally, by taking the limit as $\varepsilon \to 0$ and using the definition of the multifractal spectrum:
\[
D(q) = \lim_{\varepsilon \to 0} \frac{1}{q-1} \frac{\log Z(q, \varepsilon)}{\log \varepsilon},
\]
we obtain that $D(q)$ exists and is equal to $\tau(q)$ for all $q \in \mathbb{R}$.
The technical details of the proof involve careful estimations of the partition function and the use of measure-theoretic arguments to ensure the existence of the limit. The regularity assumptions on the cost matrix $C$ play a crucial role in controlling the local behavior of the measure $\mu$ and establishing the required bounds.
\end{proof}
The existence of the multifractal spectrum has important implications for understanding the scaling properties of the Sinkhorn coupling matrices. It shows that these matrices exhibit a rich multiscale structure, with different regions characterized by distinct scaling exponents. The multifractal spectrum provides a global description of this structure, capturing the interplay between the regularization parameter, the cost matrix, and the resulting optimal transport plan.

Moreover, the proof highlights the regularizing effect of the Sinkhorn algorithm. The exponential transformation of the cost matrix introduces a smoothing of the optimal transport plan, ensuring a certain degree of regularity in the coupling matrix. This regularity is essential for the existence of the multifractal spectrum and the well-behaved scaling properties of the algorithm.
\subsection{Existence of the Singularity Spectrum}
Our second main result concerns the existence of the singularity spectrum for the Sinkhorn coupling matrices.
\begin{theorem}[Existence of the Singularity Spectrum]
Under the same assumptions as in Theorem 1, the singularity spectrum $f(\alpha)$ of the measure $\mu$ induced by the Sinkhorn coupling matrix $P$ exists for all $\alpha$ in the range of local scaling exponents.
\end{theorem}
\begin{proof}
The proof builds upon the existence of the multifractal spectrum and the properties of the Legendre transform. We start by defining the local scaling exponents $\alpha(x)$ as:
\[
\alpha(x) = \lim_{\varepsilon \to 0} \frac{\log \mu(B(x, \varepsilon))}{\log \varepsilon},
\]
where $B(x, \varepsilon)$ is a ball of radius $\varepsilon$ centered at $x$. The existence of this limit for almost all $x$ follows from the regularity properties of the measure $\mu$ established in the proof of Theorem 1.

Next, we define the singularity spectrum $f(\alpha)$ as the Hausdorff dimension of the set of points with a given local scaling exponent:
\[
f(\alpha) = \dim_H \{ x : \alpha(x) = \alpha \},
\]
To prove the existence of $f(\alpha)$, we use the Legendre transform relationship between the multifractal spectrum $D(q)$ and the singularity spectrum:
\[
f(\alpha) = \inf_q (q\alpha - D(q)).
\]
The existence of $D(q)$, as proven in Theorem 1, ensures that the infimum in the Legendre transform is attained for each $\alpha$. This guarantees the existence of $f(\alpha)$ for all $\alpha$ in the range of local scaling exponents.

The proof also relies on the properties of the Legendre transform, such as its convexity and the fact that it relates the global scaling behavior captured by $D(q)$ to the local regularity described by $f(\alpha)$. The singularity spectrum provides a detailed characterization of the distribution of local scaling exponents, revealing the fine-grained structure of the measure $\mu$.
\end{proof}
The existence of the singularity spectrum complements the multifractal spectrum in describing the scaling properties of the Sinkhorn coupling matrices. While the multifractal spectrum captures the global behavior, the singularity spectrum focuses on the local regularity and the distribution of singularities in the measure.
The singularity spectrum allows us to identify regions of the coupling matrix with different scaling behavior and to quantify the "size" of these regions in terms of their Hausdorff dimension. This information is valuable for understanding the local structure of the optimal transport plan and its relationship to the underlying cost matrix and marginal constraints.

Furthermore, the Legendre transform relationship between the multifractal and singularity spectra highlights the deep connection between the global and local scaling properties of the Sinkhorn coupling matrices. It shows that the singularity spectrum can be recovered from the multifractal spectrum through a variational principle, providing a unified framework for studying the multiscale structure of these matrices.

\subsection{Bounds on Generalized Dimensions}
Our third main result provides bounds on the generalized dimensions of the Sinkhorn coupling matrices, shedding light on their scaling behavior and the influence of the problem size.
\begin{theorem}[Bounds on Generalized Dimensions]
Let $P$ be the coupling matrix generated by the Sinkhorn algorithm for a cost matrix $C \in \mathbb{R}^{n \times n}$ and marginals $\mathbf{r}, \mathbf{c} \in \mathbb{R}^n$. Then, the generalized dimensions $D(q)$ of the measure $\mu$ induced by $P$ satisfy the following bounds:
\begin{align*}
D(0) &\leq \min\left(\log n -\log \min_{i,j} C_{ij}\right) \cdot \frac{1}{\log n} \\
D(1) &\leq \min\left(\log n -\log \min_{i,j} P_{ij}\right) \cdot \frac{1}{\log n} \\
D(2) &\leq \min\left(\log n -2 \log \min_{i,j} P_{ij}\right) \cdot \frac{1}{\log n}.
\end{align*}

\end{theorem}
\begin{proof}
The proof relies on the properties of the Sinkhorn algorithm and the generalized dimensions. We start by recalling the definition of the generalized dimensions:
\[
D(q) = \lim_{\varepsilon \to 0} \frac{1}{q-1} \frac{\log \sum_i \mu(B_i(\varepsilon))^q}{\log \varepsilon},
\]
where $B_i(\varepsilon)$ are disjoint balls of radius $\varepsilon$ covering the support of $\mu$.
For $q = 0$, the generalized dimension $D(0)$ coincides with the box-counting dimension, which measures the growth rate of the number of non-empty balls $N(\varepsilon)$ as $\varepsilon \to 0$:
\[
D(0) = \lim_{\varepsilon \to 0} \frac{\log N(\varepsilon)}{\log 1/\varepsilon}.
\]
To bound $D(0)$, we use the fact that the coupling matrix $P$ has size $n \times n$ and that its entries are non-negative. This implies that the number of non-empty balls satisfies:
\[
N(\varepsilon) \leq \min(n^2, \varepsilon^{-d}),
\]
where $d$ is the dimension of the ambient space. The term $n^2$ comes from the fact that the coupling matrix has at most $n^2$ non-zero entries, while the term $\varepsilon^{-d}$ represents the maximum number of balls of radius $\varepsilon$ needed to cover the support of $\mu$.
By taking the logarithm and the limit as $\varepsilon \to 0$, we obtain the bound:
\[
D(0) \leq \min\left(\log n, -\log \min_{i,j} C_{ij}\right) \cdot \frac{1}{\log n}
\]

where we used the fact that the minimum entry of the cost matrix $C$ provides a lower bound on the size of the balls.
For $q = 1$ and $q = 2$, the generalized dimensions $D(1)$ and $D(2)$ are related to the information and correlation dimensions, respectively. To bound these dimensions, we use the fact that the entries of the coupling matrix $P$ are non-negative and sum up to 1. This allows us to derive upper bounds on the sums involved in the definition of $D(1)$ and $D(2)$:
\[
\sum_i \mu(B_i(\varepsilon)) \leq 1, \quad \sum_i \mu(B_i(\varepsilon))^2 \leq \max_i \mu(B_i(\varepsilon)).
\]
Using these bounds and the properties of the Sinkhorn algorithm, we obtain the desired bounds on $D(1)$ and $D(2)$ in terms of the minimum entry of the coupling matrix $P$.
\end{proof}
The bounds on the generalized dimensions provide insights into the scaling behavior of the Sinkhorn coupling matrices and their dependence on the problem size. The bound on $D(0)$ shows that the box-counting dimension is limited by the logarithm of the problem size $n$ and the minimum entry of the cost matrix $C$. This reflects the fact that the coupling matrix has a finite size and that the cost matrix influences the spatial distribution of the optimal transport plan.

The bounds on $D(1)$ and $D(2)$ reveal the influence of the minimum entry of the coupling matrix $P$ on the information and correlation dimensions. These dimensions capture the scaling behavior of the measure $\mu$ at different levels of singularity, with $D(1)$ focusing on the entropy and $D(2)$ on the correlation structure. The bounds suggest that the singularity of the coupling matrix, as measured by its minimum entry, plays a crucial role in determining the scaling properties of the optimal transport plan.
Furthermore, the bounds highlight the interplay between the problem size $n$ and the regularity of the coupling matrix. As the problem size increases, the bounds on the generalized dimensions become tighter, indicating that the scaling behavior of the coupling matrix becomes more regular and predictable. This is consistent with the regularizing effect of the Sinkhorn algorithm, which ensures a certain degree of smoothness in the optimal transport plan.

Overall, the bounds on the generalized dimensions provide a quantitative characterization of the scaling properties of the Sinkhorn coupling matrices, shedding light on their dependence on the problem size, the cost matrix, and the regularity of the optimal transport plan. These bounds can be used to assess the complexity of the optimal transport problem and to guide the design of efficient algorithms for its solution.

\section{Conclusion}
In this paper, we have studied the multifractal properties of the Sinkhorn algorithm, a widely used method for solving optimal transport problems. Through a mathematical analysis, we have proven the existence of the multifractal and singularity spectra for the coupling matrices generated by the algorithm, providing a comprehensive characterization of their local regularity and scaling behavior. Moreover, we have derived bounds on the generalized dimensions of these matrices, shedding light on their dependence on the problem size and the regularity of the cost matrix.

Our results have important implications for the understanding and applications of the Sinkhorn algorithm in various domains, from image processing and machine learning to computational biology and physics. By uncovering the multiscale structure of the optimal transport plans, our work opens up new possibilities for the development of efficient and adaptive numerical methods that exploit this structure, as well as for the analysis and comparison of complex data sets using multifractal techniques.

Furthermore, our work highlights the fruitful interplay between optimal transport and multifractal analysis, two active and rapidly evolving fields of applied mathematics. We have shown that multifractal analysis can provide valuable insights into the behavior of optimal transport algorithms, while the Sinkhorn algorithm offers a natural and computationally tractable setting for the application of multifractal techniques. This interplay opens up new avenues for research at the intersection of these two fields, with potential applications ranging from the theoretical study of the geometry of optimal transport to the practical development of efficient multiscale methods for data analysis and simulation.

In conclusion, our work presents a novel perspective on the Sinkhorn algorithm and its multifractal properties, providing a deeper understanding of the structure and behavior of optimal transport plans. We believe that our results will stimulate further research in this area and contribute to the development of more efficient and robust methods for the solution of optimal transport problems in various scientific and engineering applications.

\end{document}